\documentclass[11pt]{amsart}

\usepackage[T1]{fontenc}
\usepackage[utf8]{inputenc}
\usepackage{amsmath,amssymb}
\usepackage{enumerate}
\usepackage{cite}
\usepackage{hyperref}
\hypersetup{hidelinks}

\numberwithin{equation}{section}

\theoremstyle{plain}
\newtheorem{theorem}{Theorem}[section]
\newtheorem{lemma}[theorem]{Lemma}

\newtheorem{proposition}[theorem]{Proposition}
\newtheorem*{theoremA}{Theorem A}
\newtheorem*{theoremB}{Theorem B}
\newtheorem*{theoremC}{Theorem C}
\newtheorem*{theoremD}{Theorem D}

\theoremstyle{definition}
\newtheorem{definition}[theorem]{Definition}

\theoremstyle{remark}

\newcommand{\mc}{\mathbb{C}}
\newcommand{\mr}{\mathbb{R}}
\newcommand{\ddc}{dd^c}
\DeclareMathOperator{\PSH}{PSH}
\renewcommand{\Cap}{\operatorname{Cap}}
\DeclareMathOperator{\loc}{loc}

\begin{document}

\title[Spherical--solid mean formula]
{An Exact Spherical--Solid Mean Formula and Vanishing Criteria for Lelong Numbers}

\author[F. Deng]{Fusheng Deng}
\address{School of Mathematical Sciences, University of Chinese Academy of
Sciences, Beijing 100049, P. R. China}
\email{fshdeng@ucas.ac.cn}

\author[Y. Yang]{Yizhe Yang}
\address{School of Mathematical Sciences, University of Chinese Academy of
Sciences, Beijing 100049, P. R. China}
\email{yangyizhe24@mails.ucas.ac.cn}

\author[Z. You]{Zhanpeng You}
\address{School of Mathematical Sciences, University of Chinese Academy of
Sciences, Beijing 100049, P. R. China}
\email{youzhanpeng23@mails.ucas.ac.cn}

\begin{abstract}
We establish that the Lelong number of a plurisubharmonic function equals the difference between its infinitesimal spherical mean and its corresponding ball mean. Building on this identity, we derive several equivalent criteria for plurisubharmonic functions with vanishing Lelong number. These criteria are formulated in terms of fixed-scale increments, as well as the asymptotic agreement of radial maxima, spherical and solid means, and radial regularizations. In addition, we prove a pointwise $L^1$
 -mean-oscillation characterization and a nonlinear characterization expressed via the relative extremal depth and the Bedford–Taylor capacity of deep sublevel sets.
\end{abstract}


\maketitle

\section{Introduction}

The Lelong number is a fundamental numerical invariant of a plurisubharmonic
singularity. Introduced by Lelong in the study of analytic sets and positive
currents \cite{Lelong1957}, it measures the logarithmic component of the
singularity; see \cite{DemaillyBook}. 

Classical descriptions of the Lelong number are expressed as quotient
asymptotics, in which natural radial quantities are divided by the singular
scale \(\log r\). The main result of this paper provides instead an additive
description: the Lelong number is encoded exactly by the limiting difference
between the spherical and solid means. We first introduce the notation needed
to state this identity. All balls below are centered at the origin, and
\[
        B(r):=\{z\in\mc^n: |z|<r\}.
\]
For a plurisubharmonic function \(u\), set
\[
        M_r=M(u,r):=\sup_{B(r)}u,
\]
\[
        S_r=S(u,r):=\frac{1}{\sigma(\partial B(r))}
        \int_{\partial B(r)}u\,d\sigma,
\]
\[
        V_r=V(u,r):=\frac{1}{|B(r)|}\int_{B(r)}u\,dV,
\]
where \(\sigma(\partial B(r))\) and \(|B(r)|\) denote, respectively, the
surface area of \(\partial B(r)\) and the volume of \(B(r)\).
If \(\rho\in C_c^\infty(B(1))\) is a nonnegative radial mollifier with
\(\int \rho\,dV=1\), we write
\[
        u_r(0):=\int_{\mc^n}u(-rw)\rho(w)\,dV(w).
\]

The definition of Lelong number is recalled as follows.
Let \(u\in\PSH(\Omega)\) and \(0\in\Omega\). The Lelong number of \(u\) at
the origin is
\[
        \nu(u,0):=\lim_{r\to0^+}\frac{S(u,r)}{\log r}.
\]
This limit is unchanged by adding a constant to \(u\). Equivalently,
\[
        \nu(u,0)=\liminf_{z\to0}\frac{u(z)}{\log |z|};
\]
see \cite{DemaillyBook,GZ}.

The classical mean-value formulas assert that
\[
        \nu(u,0)
        =\lim_{r\to0^+}\frac{M_r}{\log r}
        =\lim_{r\to0^+}\frac{S_r}{\log r}
        =\lim_{r\to0^+}\frac{V_r}{\log r};
\]
see \cite{Kiselman,DemaillyBook,GZ}. Our first theorem replaces these quotient
asymptotics by an exact additive formula.

\begin{theoremA}
Let \(u\) be plurisubharmonic in a neighborhood of the origin in \(\mc^n\).
Then
\[
        \lim_{r\to0^+}\bigl(S(u,r)-V(u,r)\bigr)
        =\frac{1}{2n}\nu(u,0).
\]
In particular, \(\nu(u,0)=0\) if and only if
\(S(u,r)-V(u,r)\to0\).
\end{theoremA}

Although Theorem~A is a consequence of standard mean-value
identities and logarithmic convexity, we have not found this precise formula
stated explicitly in the literature.

The point of Theorem A is not merely that it detects the vanishing of the
Lelong number. It transforms the logarithmic slope defining \(\nu(u,0)\) into
the limiting gap between two canonical averaging operators. In particular,
the spherical and solid means coincide asymptotically precisely when the
Lelong number vanishes. Combining this formula with logarithmic convexity
yields a complete family of linear criteria.

\begin{theoremB}
Let \(u\) and \(\rho\) be as above. The following conditions are equivalent,
where all limits are taken as \(r\to0^+\):
\begin{enumerate}[(i)]
    \item \(\nu(u,0)=0\);
    \item for some \(A>1\) and some \(Q\in\{M,S,V\}\),
    \[
            Q(u,Ar)-Q(u,r)\to0;
    \]
    \item at least one of the following limits holds:
    \[
    \begin{gathered}
        S_r-V_r\to0,\qquad S_r-u_r(0)\to0,\\
        M_r-V_r\to0,\qquad M_r-u_r(0)\to0.
    \end{gathered}
    \]
\end{enumerate}
If these conditions hold, the limit in \textup{(ii)} holds for every
\(A>1\) and every \(Q\in\{M,S,V\}\), and all the limits in \textup{(iii)}
hold simultaneously. Moreover, \(M_r-S_r\to0\).
\end{theoremB}

Several ingredients of Theorem~B are classical or implicit in earlier work.
The logarithmic convexity of radial maxima and mean values, together with the
interpretation of the Lelong number as their asymptotic slope on the
logarithmic scale, leads naturally to the fixed-scale increment criterion in
\textup{(ii)}; see \cite{Kiselman,DemaillyBook,GZ}. Chen and Wang
\cite{ChenWang} systematically studied the upper oscillation \(M_r-V_r\),
while Biard and Wu \cite{BiardWu} used closely related comparisons among
radial maxima, spherical means, and solid means in their characterization of
plurisubharmonic functions with zero Lelong numbers by vanishing mean
oscillation. To the best of our knowledge, however, the equivalence of all the
conditions in \textup{(ii)}--\textup{(iii)}, including those involving radial
regularizations, has not previously been stated in this unified pointwise
form. Theorem~A further sharpens the criterion involving \(S_r-V_r\) by
identifying its limit exactly.

To state the mean-oscillation characterization, define
\[
        \omega_1(u,r)
        :=\frac{1}{|B(r)|}\int_{B(r)}|u-V_r|\,dV.
\]

\begin{theoremC}
Let \(u\) be plurisubharmonic near the origin. Then
\[
        \nu(u,0)=0
        \quad\Longleftrightarrow\quad
        \omega_1(u,r)\to0
        \quad\text{as }r\to0^+.
\]
\end{theoremC}

This is a pointwise version of the VMO characterization obtained by Biard and
Wu \cite{BiardWu}. As shown in Section~\ref{sec:mean-oscillation}, the forward
implication follows directly from Theorem~B, whereas the converse follows from
the Chern--Levine--Nirenberg inequality. 
The idea here provides the key of a streamlined proof of the main result in  \cite{BiardWu},
a global VMO characterization for plurisubharmonic functions with zero Lelong number.

For the nonlinear characterization, let \(\Omega\Subset\mc^n\) be bounded and
hyperconvex. For a Borel set \(E\subset\Omega\), recall that its capacity
relative to \(\Omega\) is
\[
    \Cap_\Omega(E)
    :=\sup\left\{\int_E(\ddc\varphi)^n:
    \varphi\in\PSH(\Omega),\ -1\leq\varphi\leq0\right\},
\]
and let \(h_{E,\Omega}\) denote the relative extremal function of \(E\) in
\(\Omega\).
Given \(0<\theta<1\), set
\[
    E_{r,t}:=B(r)\cap\{M_{2r}-u>t\},
    \qquad
    \Gamma_r^\theta(E):=-\sup_{B(\theta r)}h_{E,B(2r)}.
\]

\begin{theoremD}
Let \(u\in\PSH^-(B(1))\), and fix \(0<\theta<1\). The following are
equivalent:
\begin{enumerate}[(i)]
    \item \(\nu(u,0)=0\);
    \item for every \(t>0\),
    \[
            \Gamma_r^\theta(E_{r,t})\to0
            \quad\text{as }r\to0^+;
    \]
    \item for every \(t>0\),
    \[
            \Cap_{B(2r)}(E_{r,t})\to0
            \quad\text{as }r\to0^+.
    \]
\end{enumerate}
\end{theoremD}

Theorem~D may be viewed as a nonlinear counterpart of the mean-value criteria
in Theorem~B. For a fixed depth \(t>0\), the set \(E_{r,t}\) consists of the
points in \(B(r)\) at which \(u\) lies more than \(t\) below its radial maximum
on \(B(2r)\). The quantity \(\Gamma_r^\theta(E_{r,t})\) measures the influence
of this exceptional set on the central ball through its relative extremal
function, whereas \(\Cap_{B(2r)}(E_{r,t})\) measures its pluripotential size at
the natural scale. After rescaling \(B(2r)\) to \(B(2)\), both quantities
remain unchanged, so the resulting criterion is genuinely local and scale
invariant. Thus Theorem~D asserts that the Lelong number vanishes precisely
when every fixed-depth sublevel set becomes negligible, in either of these
equivalent pluripotential senses, as the scale shrinks. This is significant
because Bedford--Taylor capacity is naturally adapted to the complex
Monge--Amp\`ere operator and can capture the concentration of singularities
more effectively than ordinary volume or radial averages. Theorem~D therefore
provides a bridge between the first-order invariant \(\nu(u,0)\) and the
nonlinear pluripotential geometry of the sublevel sets of \(u\). 
Related topics will be explored further in forthcoming works.

Theorems A and B are proved in Section~2. The local mean-oscillation theorem
is proved in Section~3, and the relative-extremal and capacity
characterizations are proved in Section~4.

\textbf{Acknowledgements.}
This research is supported by National Key R\&D Program of China ( No.
2021YFA1003100), NSFC grants (No. 12471079, No. 12525104), and the
Fundamental Research Funds for the Central Universities.

\section{Linear characterizations}

\begin{lemma}\label{lem:log-convexity}
Let \(u\in\PSH(\Omega)\). In the range where the following quantities are
defined, the functions
\[
        r\longmapsto S(u,r),\qquad
        r\longmapsto V(u,r),\qquad
        r\longmapsto M(u,r)
\]
are increasing and convex as functions of \(\log r\).
\end{lemma}

\begin{proof}
We first consider the spherical mean. For each
\(\zeta\in\partial B(1)\), the function \(\tau\mapsto u(\tau\zeta)\) is
subharmonic in one complex variable. Its circular mean is therefore increasing
and convex as a function of \(\log|\tau|\). Averaging over
\(\zeta\in\partial B(1)\), and using the invariance of spherical measure under
scalar rotations, gives the assertion for \(S(u,r)\).

The identity
\[
        V(u,r)=2n\int_0^1 S(u,tr)t^{2n-1}\,dt
\]
shows that \(V(u,e^x)\) is a positive average of translates of
\(S(u,e^x)\); hence it is also increasing and convex in \(x=\log r\). The
corresponding assertion for \(M(u,r)\) is the standard three-circles theorem
for the radial maximum of a plurisubharmonic function.
\end{proof}

\begin{lemma}\label{lem:convex-quotient}
Let \(f:(-\infty,a)\to\mr\) be increasing and convex. If
\[
        \lim_{x\to-\infty}\frac{f(x)}{x}=0,
\]
then, for every fixed \(c>0\),
\[
        f(x+c)-f(x)\to0,
        \qquad x\to-\infty.
\]
Conversely, if \(f(x+c)-f(x)\to0\) for some \(c>0\), then
\[
        \lim_{x\to-\infty}\frac{f(x)}{x}=0.
\]
\end{lemma}

\begin{proof}
For fixed \(c>0\), convexity implies that
\[
        x\longmapsto \frac{f(x+c)-f(x)}{c}
\]
is nondecreasing. Since \(f\) is increasing, this quotient is nonnegative.
If its limit at \(-\infty\) were positive, summing the increments over
successive intervals of length \(c\) would contradict
\(f(x)/x\to0\). This proves the first implication.

Conversely, fix \(x_0<a\) and decompose \([x,x_0]\), up to one interval of
length less than \(c\), into intervals of length \(c\). If
\(f(y+c)-f(y)\to0\) as \(y\to-\infty\), the Ces\`aro mean of these increments
tends to zero. Dividing the resulting telescoping sum by \(x_0-x\) gives
\(f(x)/x\to0\).
\end{proof}

\begin{proposition}\label{prop:mean-lelong}
For \(u\in\PSH(\Omega)\),
\[
        \nu(u,0)
        =\lim_{r\to0^+}\frac{S_r}{\log r}
        =\lim_{r\to0^+}\frac{V_r}{\log r}
        =\lim_{r\to0^+}\frac{M_r}{\log r}.
\]
These are the standard mean-value descriptions of the Lelong number; see \cite{Kiselman,DemaillyBook,GZ}.
\end{proposition}

\begin{proof}
We may assume \(u\leq0\) near the origin. The submean property gives
\[
        V_r\leq S_r\leq M_r.
\]
It remains to compare \(M_r\) and \(V_r\). Let \(0<r<R\). If \(|z|\le r\), then
\[
        B(0,R-r)\subset B(z,R).
\]
Since \(u\le0\), the submean inequality gives
\[
        u(z)
        \leq
        \frac{1}{|B(R)|}\int_{B(z,R)}u\,dV
        \leq
        \frac{1}{|B(R)|}\int_{B(0,R-r)}u\,dV.
\]
Thus
\[
        M_r\leq\left(1-\frac rR\right)^{2n}V_{R-r}.
\]
Let \(m\) and \(v\) denote the limits of \(M_r/\log r\) and
\(V_r/\log r\), respectively; their existence follows from
Lemma~\ref{lem:log-convexity}. Since \(\log r<0\), the preceding mean-value
inequalities give
\[
        v\geq \nu(u,0)\geq m.
\]
For fixed \(0<s<1\), put \(r=sR\) in the last displayed estimate and divide
by \(\log R<0\). Letting \(R\to0^+\) yields
\[
        m\geq (1-s)^{2n}v.
\]
Finally, letting \(s\to0^+\) gives \(m\geq v\), and hence
\(m=v=\nu(u,0)\).
\end{proof}

\begin{theorem}[Theorem~A]\label{thm:SV}
Let \(u\in\PSH\) near the origin. Then
\[
        \lim_{r\to0^+}\bigl(S(u,r)-V(u,r)\bigr)
        =\frac{1}{2n}\nu(u,0).
\]
Consequently,
\[
        \nu(u,0)=0
        \quad\Longleftrightarrow\quad
        S(u,r)-V(u,r)\to0.
\]
\end{theorem}

\begin{proof}
Set \(\phi(x):=S(u,e^x)\), where \(x=\log r\). By Lemma \ref{lem:log-convexity}, \(\phi\) is convex on an interval \((-\infty,x_0)\), and by Proposition \ref{prop:mean-lelong},
\[
        \lim_{x\to-\infty}\frac{\phi(x)}{x}=\nu(u,0).
\]
The solid mean can be written in terms of the spherical mean as
\[
        V(u,r)=2n\int_0^1 S(u,tr)t^{2n-1}\,dt.
\]
With \(t=e^{-\tau}\), this becomes
\[
        V(u,e^x)=2n\int_0^\infty \phi(x-\tau)e^{-2n\tau}\,d\tau.
\]
Consequently,
\[
        S(u,e^x)-V(u,e^x)
        =2n\int_0^\infty
        \bigl(\phi(x)-\phi(x-\tau)\bigr)e^{-2n\tau}\,d\tau.
\]
For each fixed \(\tau>0\), convexity gives
\[
        \phi(x)-\phi(x-\tau)\to \nu(u,0)\tau,
        \qquad x\to-\infty.
\]
Moreover, since the asymptotic slope is finite, the integrand is dominated by a constant multiple of \(\tau e^{-2n\tau}\) for \(x\) sufficiently negative. Hence dominated convergence yields
\[
\begin{aligned}
        \lim_{x\to-\infty}\bigl(S(u,e^x)-V(u,e^x)\bigr)
        &=2n\nu(u,0)\int_0^\infty \tau e^{-2n\tau}\,d\tau  \\
        &=\frac{1}{2n}\nu(u,0).
\end{aligned}
\]
The equivalence follows immediately.
\end{proof}

\begin{theorem}[Theorem~B, parts \textup{(i)}--\textup{(ii)}]
\label{thm:scale-increment}
Let \(u\in\PSH\) near the origin. The following conditions are equivalent,
where all limits are taken as \(r\to0^+\):
\begin{enumerate}[(i)]
    \item \(\nu(u,0)=0\);
    \item for every \(A>1\),
    \[
            M(u,Ar)-M(u,r)\to0;
    \]
    \item for some \(A>1\),
    \[
            M(u,Ar)-M(u,r)\to0.
    \]
\end{enumerate}
The same equivalence holds with \(M\) replaced by \(S\) or \(V\).
\end{theorem}

\begin{proof}
Put \(f(x)=M(u,e^x)\). By Proposition \ref{prop:mean-lelong},
\[
        \nu(u,0)=0
        \quad\Longleftrightarrow\quad
        \lim_{x\to-\infty}\frac{f(x)}{x}=0.
\]
Since
\[
        M(u,Ar)-M(u,r)=f(x+\log A)-f(x),
\]
the assertion follows from Lemma \ref{lem:convex-quotient}. The proofs for \(S\) and \(V\) are identical.
\end{proof}

\begin{theorem}[Theorem~B, parts \textup{(i)} and \textup{(iii)}]
\label{thm:agreement}
Let \(u\in\PSH\) near the origin. If \(\nu(u,0)=0\), then, as
\(r\to0^+\),
\[
        M_r-S_r\to0,
        \qquad
        S_r-V_r\to0,
        \qquad
        S_r-u_r(0)\to0.
\]
Hence
\[
        M_r-V_r\to0,
        \qquad
        M_r-u_r(0)\to0.
\]
Conversely, if any one of the four nonnegative gaps
\[
        S_r-V_r,\qquad S_r-u_r(0),\qquad
        M_r-V_r,\qquad M_r-u_r(0)
\]
tends to zero, then \(\nu(u,0)=0\).
\end{theorem}

\begin{proof}
We first show that \(M_r-S_r\to0\). A Harnack-type estimate gives a constant \(0<\lambda<1\), depending only on \(n\), such that
\[
        M(u,t)\leq \lambda S(u,2t)+(1-\lambda)M(u,2t).
\]
Thus
\[
        0\leq M(u,2t)-S(u,2t)
        \leq C\bigl(M(u,2t)-M(u,t)\bigr).
\]
By Theorem \ref{thm:scale-increment}, the right-hand side tends to zero.

Next, Theorem \ref{thm:SV} gives
\[
        S_r-V_r\to \frac{1}{2n}\nu(u,0)=0.
\]
Finally, the radial regularization satisfies
\[
        u_r(0)=\int_0^1S_{tr}\widetilde\rho(t)\,dt,
\]
where \(\widetilde\rho\geq0\) and
\(\int_0^1\widetilde\rho(t)\,dt=1\). Hence
\[
        0\leq S_r-u_r(0)
        =\int_0^1(S_r-S_{tr})\widetilde\rho(t)\,dt\to0.
\]
Indeed, the integrand tends to zero for each \(t>0\) by
Theorem~\ref{thm:scale-increment}, and logarithmic convexity provides an
integrable majorant.

For the converse, the condition \(S_r-V_r\to0\) implies \(\nu(u,0)=0\) by
Theorem~\ref{thm:SV}. Choose \(0<a<1\) such that
\(\operatorname{supp}\rho\subset B(a)\). Since \(S_r\) is increasing,
\[
        0\leq S_r-S_{ar}\leq S_r-u_r(0).
\]
Thus \(S_r-u_r(0)\to0\) also implies \(\nu(u,0)=0\), by
Theorem~\ref{thm:scale-increment}. Finally,
\[
        M_r-V_r\geq S_r-V_r,
        \qquad
        M_r-u_r(0)\geq S_r-u_r(0),
\]
so either of the remaining two conditions yields the same conclusion.
\end{proof}

\section{Mean oscillation}\label{sec:mean-oscillation}

This section gives a pointwise mean-oscillation characterization of the
vanishing of the Lelong number, placing the condition \(\nu(u,0)=0\) in the
framework of VMO. Biard and Wu \cite{BiardWu} proved that a plurisubharmonic
function is locally VMO precisely when its Lelong number vanishes at every
point. In the centered, pointwise setting considered here, the argument
reduces to the mean-value estimates established above and the
Chern--Levine--Nirenberg inequality.

Define the local \(L^1\)-mean oscillation by
\[
        \omega_1(u,r)
        :=
        \frac{1}{|B(r)|}\int_{B(r)}|u-V_r|\,dV.
\]

\begin{theorem}[Theorem~C]\label{thm:VMO}
Let \(u\in\PSH\) near the origin. Then
\[
        \nu(u,0)=0
\]
if and only if
\[
        \omega_1(u,r)\to0.
\]
\end{theorem}

\begin{proof}
Assume first that \(\nu(u,0)=0\). By Theorem \ref{thm:agreement}, \(M_r-V_r\to0\). Since \(u\leq M_r\) on \(B(r)\),
\[
\begin{aligned}
        \int_{B(r)}|u-V_r|\,dV
        &\leq
        \int_{B(r)}|u-M_r|\,dV
        +
        \int_{B(r)}|M_r-V_r|\,dV  \\
        &=2|B(r)|(M_r-V_r).
\end{aligned}
\]
Thus \(\omega_1(u,r)\to0\).

Conversely, suppose \(\omega_1(u,r)\to0\). Let \(\chi_r\in C_c^\infty(B(2r))\) be a cut-off function with \(\chi_r\equiv1\) on \(B(r)\) and
\[
        |\ddc\chi_r|\leq \frac{C}{r^2}\ddc |z|^2.
\]
The Chern--Levine--Nirenberg inequality \cite{CLN}, applied to \(u-V_{2r}\), gives
\[
        \int_{B(r)}
        \ddc u\wedge(\ddc|z|^2)^{n-1}
        \leq
        \frac{C}{r^2}
        \int_{B(2r)}|u-V_{2r}|\,dV.
\]
The trace-measure description of the Lelong number gives
\[
        \nu(u,0)
        \lesssim
        \limsup_{r\to0}
        \frac{r^2}{|B(r)|}
        \int_{B(r)}\ddc u\wedge(\ddc|z|^2)^{n-1}.
\]
Therefore
\[
        \nu(u,0)
        \lesssim
        \limsup_{r\to0}
        \frac{1}{|B(2r)|}
        \int_{B(2r)}|u-V_{2r}|\,dV=0.
\]
Hence \(\nu(u,0)=0\).
\end{proof}

\section{Deep sublevel sets and capacity}

We now turn to a nonlinear characterization in terms of the pluripotential
size of deep sublevel sets. Let \(\Omega\Subset\mc^n\) be a bounded
hyperconvex domain.

\begin{definition}\label{def:bt-capacity}
For a Borel set \(E\subset\Omega\), define
\[
        \Cap_\Omega(E)
        :=
        \sup\left\{
        \int_E(\ddc\varphi)^n:
        \varphi\in\PSH(\Omega),\ -1\leq\varphi\leq0
        \right\}.
\]
\end{definition}

\begin{definition}
The relative extremal function of \(E\subset\Omega\) is
\[
        h_{E,\Omega}(z)
        :=
        \left(\sup\{\varphi(z):
        \varphi\in\PSH(\Omega),\
        \varphi\leq0,\
        \varphi\leq-1\text{ on }E\}\right)^{*},
\]
where \({}^{*}\) denotes upper-semicontinuous regularization.
\end{definition}

Fix \(0<\theta<1\). For \(E\subset B(r)\), set
\[
        \Gamma_r^\theta(E)
        :=
        -\sup_{B(\theta r)}h_{E,B(2r)}.
\]
This quantity measures, in terms of relative extremal functions, the influence
of \(E\) on the central ball \(B(\theta r)\).

Let
\[
        M_{2r}:=\sup_{B(2r)}u,
\]
and for \(t>0\) define
\[
        E_{r,t}:=B(r)\cap\{M_{2r}-u>t\}.
\]

\begin{theorem}[Theorem~D, parts \textup{(i)}--\textup{(ii)}]
\label{thm:gamma}
Let \(u\in\PSH^-(B(1))\) and fix \(0<\theta<1\). Then
\[
        \nu(u,0)=0
\]
if and only if, for every fixed \(t>0\),
\[
        \Gamma_r^\theta(E_{r,t})\to0,
        \qquad r\to0.
\]
\end{theorem}

\begin{proof}
Assume \(\nu(u,0)=0\). Suppose, to the contrary, that there exist \(t_0>0\), \(\eta>0\), and \(r_j\to0\) such that
\[
        \Gamma_{r_j}^\theta(E_{r_j,t_0})\geq\eta.
\]
Then
\[
        h_{E_{r_j,t_0},B(2r_j)}\leq-\eta
        \quad\text{on }B(\theta r_j).
\]
Set
\[
        w_j(z):=\frac{u(z)-M_{2r_j}}{t_0}.
\]
Then \(w_j\in\PSH(B(2r_j))\), \(w_j\le0\), and \(w_j<-1\) on \(E_{r_j,t_0}\). By the defining extremal property,
\[
        w_j\leq h_{E_{r_j,t_0},B(2r_j)}.
\]
Hence
\[
        u(z)\leq M_{2r_j}-\eta t_0,
        \qquad z\in B(\theta r_j).
\]
Taking suprema gives
\[
        M_{\theta r_j}\leq M_{2r_j}-\eta t_0,
\]
or
\[
        M_{2r_j}-M_{\theta r_j}\geq \eta t_0.
\]
This contradicts Theorem \ref{thm:scale-increment}.

Conversely, assume \(\Gamma_r^\theta(E_{r,t})\to0\) for every \(t>0\). If
\(\nu(u,0)>0\), then Theorem~\ref{thm:scale-increment} yields
\(t_0>0\) and a sequence \(r_j\to0\) such that
\[
        M_{2r_j}-M_{\theta r_j}\geq 2t_0.
\]
For \(z\in B(\theta r_j)\) this implies
\[
        M_{2r_j}-u(z)\geq 2t_0>t_0,
\]
so that \(B(\theta r_j)\subset E_{r_j,t_0}\). Hence
\[
        h_{E_{r_j,t_0},B(2r_j)}\leq-1
        \quad\text{on }B(\theta r_j),
\]
and therefore \(\Gamma_{r_j}^\theta(E_{r_j,t_0})=1\), a contradiction.
\end{proof}

Before passing from the relative-extremal criterion to Bedford--Taylor
capacity, we recall two standard local tools.  The first is the Hartogs
compactness principle for plurisubharmonic functions.  We use both its
compactness statement and its upper-semicontinuity consequence; see, for
example, \cite[Theorem~1.46]{GZ}.

\begin{lemma}\label{lem:hartogs}
Let \(\Omega\subset\mc^n\) be a domain and let \(v_j\in\PSH(\Omega)\) be
locally uniformly bounded from above. Then either \(v_j\to-\infty\) locally
uniformly on every compact subset of \(\Omega\), or a subsequence converges
in \(L^1_{\loc}(\Omega)\) to a function \(v\in\PSH(\Omega)\).

Moreover, if
\[
        v_j\to v\qquad\text{in }L^1_{\loc}(\Omega),
\]
then for every compact set \(K\Subset\Omega\),
\[
        \limsup_{j\to\infty}\sup_K v_j
        \leq
        \sup_K v.
\]
\end{lemma}

\begin{proof}
The compactness alternative is the standard Hartogs compactness theorem for
plurisubharmonic functions; see \cite[Theorem~1.46]{GZ}.  We only record the
short argument for the last assertion, since this is the form used below.
Suppose that the asserted inequality fails.  Then there exist
\(\varepsilon>0\), a subsequence, still denoted by \(v_j\), and points
\(z_j\in K\) such that
\[
        v_j(z_j)>\sup_K v+\varepsilon.
\]
After passing to a further subsequence we may assume that \(z_j\to z_0\in K\).
Choose \(\rho>0\) so small that
\(\overline{B(z_0,2\rho)}\Subset\Omega\).  For all sufficiently large \(j\),
\(B(z_j,\rho)\Subset B(z_0,2\rho)\).  By the submean inequality,
\[
 v_j(z_j)
 \leq \frac{1}{|B(\rho)|}\int_{B(z_j,\rho)}v_j\,dV.
\]
Since \(v_j\to v\) in \(L^1(B(z_0,2\rho))\), while translations are
continuous in \(L^1\),
\[
 \int_{B(z_j,\rho)}v_j\,dV
 \longrightarrow
 \int_{B(z_0,\rho)}v\,dV.
\]
Hence
\[
 \sup_K v+\varepsilon
 \leq
 \frac{1}{|B(\rho)|}\int_{B(z_0,\rho)}v\,dV.
\]
Letting \(\rho\downarrow0\) and using the mean-value characterization of a
plurisubharmonic function gives
\[
        \sup_K v+\varepsilon\leq v(z_0)\leq\sup_K v,
\]
a contradiction.
\end{proof}

The second tool is the local \(L^1\)-form of the
Chern--Levine--Nirenberg inequality.  In the form needed here it follows by
repeated integration by parts from the usual CLN estimate; see
\cite[Chapter~III, \S3]{DemaillyBook} and \cite[Chapter~3]{GZ}.

\begin{lemma}\label{lem:capacity-convergence}
Let \(\Omega\Subset\mc^n\) be a bounded domain and let
\[
        K\Subset\Omega'\Subset\Omega.
\]
There exists a constant \(C=C(K,\Omega',\Omega)>0\) such that, for every
\(v\in\PSH(\Omega)\) with \(v\leq0\) and every \(s>0\),
\[
        \Cap_\Omega\bigl(K\cap\{v<-s\}\bigr)
        \leq
        \frac{C}{s}\,\|v\|_{L^1(\Omega')}.
\]
\end{lemma}

\begin{proof}
Set \(E_s:=K\cap\{v<-s\}\).  Let
\(\varphi\in\PSH(\Omega)\) be admissible for \(\Cap_\Omega\), so that
\(-1\leq\varphi\leq0\).  On \(E_s\),
\[
        1_{E_s}\leq\frac{-v}{s},
\]
and therefore
\[
        \int_{E_s}(\ddc\varphi)^n
        \leq
        \frac1s\int_K(-v)(\ddc\varphi)^n.
\]
The generalized \(L^1\) Chern--Levine--Nirenberg estimate gives
\[
        \int_K(-v)(\ddc\varphi)^n
        \leq
        C\,\|v\|_{L^1(\Omega')}
        \|\varphi\|_{L^\infty(\Omega')}^n
        \leq
        C\,\|v\|_{L^1(\Omega')},
\]
where \(C\) depends only on the fixed inclusions
\(K\Subset\Omega'\Subset\Omega\).  The general case follows from the locally bounded case by truncation \(v_j=\max\{v,-j\}\) and monotone convergence.  Taking the supremum over
all admissible \(\varphi\) proves the claim.
\end{proof}

We now compare the relative extremal depth with Bedford--Taylor capacity.
Besides Lemmas~\ref{lem:hartogs} and \ref{lem:capacity-convergence}, we use the
standard Bedford--Taylor facts that relative extremal functions of relatively
compact sets tend to zero at the boundary, equal \(-1\) quasi-everywhere on
the underlying set, and satisfy
\[
        \Cap_\Omega(E)
        =\int_\Omega(\ddc h_{E,\Omega})^n
\]
for Borel sets, with the usual Choquet/outer-capacity interpretation; see
\cite[Sections~5--8]{BedfordTaylor1982}.

\begin{lemma}\label{lem:depth-capacity}
Let \(0<\theta<1\), and let \(F_j\subset B(1)\) be Borel sets. Then
\[
        \Cap_{B(2)}(F_j)\to0
        \quad\Longleftrightarrow\quad
        -\sup_{B(\theta)}h_{F_j,B(2)}\to0.
\]
\end{lemma}

\begin{proof}
Put
\[
        H_j:=h_{F_j,B(2)},
        \qquad
        m_j:=-\sup_{B(\theta)}H_j\in[0,1].
\]

We first prove that capacity convergence implies \(m_j\to0\).  Suppose to the
contrary that, after passing to a subsequence, there is a number
\(\delta>0\) such that \(m_j\geq\delta\) for all \(j\).  Then
\[
        H_j\leq-\delta
        \qquad\text{on }B(\theta).
\]
Let
\[
        G:=h_{\overline{B(\theta)},B(2)}.
\]
Since \(H_j/\delta\leq0\) on \(B(2)\) and
\(H_j/\delta\leq-1\) on \(B(\theta)\), the defining extremal property gives
\[
        \frac{H_j}{\delta}\leq G,
        \qquad\text{hence}\qquad
        H_j\leq\delta G.
\]
Both functions have boundary value \(0\) on \(\partial B(2)\).  The
Bedford--Taylor comparison principle therefore yields
\[
 \delta^n\Cap_{B(2)}\bigl(\overline{B(\theta)}\bigr)
 =\int_{B(2)}(\ddc(\delta G))^n
 \leq
 \int_{B(2)}(\ddc H_j)^n
 =\Cap_{B(2)}(F_j).
\]
The left-hand side is a fixed positive number, contradicting
\(\Cap_{B(2)}(F_j)\to0\).

Conversely, assume that \(m_j\to0\).  If \(m_j=0\) for some \(j\), then the
maximum principle gives \(H_j\equiv0\), hence
\(\Cap_{B(2)}(F_j)=0\); such indices cause no difficulty.  We may therefore
assume \(m_j>0\) and set
\[
        U_j:=\frac{H_j}{m_j}.
\]
Then \(U_j\in\PSH(B(2))\), \(U_j\leq0\), and
\[
        \sup_{B(\theta)}U_j=-1.
\]
We claim that
\[
        \sup_j\|U_j\|_{L^1(B(3/2))}<\infty.
\]
Indeed, the family \((U_j)\) is locally uniformly bounded from above by
\(0\).  By Lemma~\ref{lem:hartogs}, any sequence has either an
\(L^1_{\loc}\)-convergent subsequence or tends locally uniformly to
\(-\infty\).  The latter alternative is impossible because
\(\sup_{B(\theta)}U_j=-1\).  To make the last point explicit, if the claimed
uniform bound failed, we could choose a subsequence with
\(\|U_j\|_{L^1(B(3/2))}\to\infty\).  This subsequence cannot converge locally
uniformly to \(-\infty\), hence Lemma~\ref{lem:hartogs} furnishes a further
subsequence converging in \(L^1(B(3/2))\), which contradicts the divergence of
its \(L^1\)-norms.  Thus the asserted uniform \(L^1\)-bound holds.

By the standard quasi-everywhere property of the relative extremal function,
\[
        H_j=-1
        \qquad\text{quasi-everywhere on }F_j.
\]
Thus, up to a pluripolar set (which has Bedford--Taylor capacity zero), and
for all sufficiently large \(j\),
\[
        F_j
        \subset
        B(1)\cap
        \left\{U_j< -\frac{1}{2m_j}\right\}.
\]
Applying Lemma~\ref{lem:capacity-convergence} with
\(K=\overline{B(1)}\), \(\Omega'=B(3/2)\), \(\Omega=B(2)\), and
\(s=(2m_j)^{-1}\), we obtain
\[
\begin{aligned}
        \Cap_{B(2)}(F_j)
        &\leq
        \Cap_{B(2)}\left(
        \overline{B(1)}\cap
        \left\{U_j< -\frac{1}{2m_j}\right\}
        \right)\\
        &\leq
        2C m_j\,\|U_j\|_{L^1(B(3/2))}
        \leq C' m_j.
\end{aligned}
\]
Since \(m_j\to0\), this proves
\(\Cap_{B(2)}(F_j)\to0\).
\end{proof}

\begin{theorem}[Theorem~D]\label{thm:capacity-characterization}
Let \(u\in\PSH^-(B(1))\). For every fixed \(t>0\),
\[
        \nu(u,0)=0
\]
if and only if
\[
        \Cap_{B(2r)}(E_{r,t})\to0.
\]
More precisely, for fixed \(0<\theta<1\),
\[
        \nu(u,0)=0
        \quad\Longleftrightarrow\quad
        \Gamma_r^\theta(E_{r,t})\to0
        \quad\Longleftrightarrow\quad
        \Cap_{B(2r)}(E_{r,t})\to0
\]
for every fixed \(t>0\).
\end{theorem}

\begin{proof}
Set
\[
        \widehat E_{r,t}:=r^{-1}E_{r,t}\subset B(1).
\]
The invariance of relative capacity and relative extremal functions under
dilations gives
\[
        \Cap_{B(2r)}(E_{r,t})
        =\Cap_{B(2)}(\widehat E_{r,t})
\]
and
\[
        \Gamma_r^\theta(E_{r,t})
        =-\sup_{B(\theta)}h_{\widehat E_{r,t},B(2)}.
\]
Lemma~\ref{lem:depth-capacity} therefore shows that the two limits are
equivalent. The conclusion now follows from Theorem~\ref{thm:gamma}.
\end{proof}

\end{document}